\documentclass{amsart}

\usepackage{graphicx}

\newtheorem{theorem}{Theorem}[section]
\newtheorem{lemma}[theorem]{Lemma}

\theoremstyle{definition}
\newtheorem{definition}[theorem]{Definition}

\newtheorem{corollary}{Corollary}[theorem]

\theoremstyle{remark}

\numberwithin{equation}{section}

\begin{document}

\title{Proof of bijection for combinatorial number system}

\author{Abu Bakar Siddique}

\address{Department of Electrical Engineering, UET Lahore}
\curraddr{NWN Lab, Al-Khwarizmi Institute of Computer Science, UET Lahore}
\email{mabs239@gmail.com}
\thanks{The first author was supported by HEC, National ICT R\&D fund.}

\author{Saadia Farid}

\address{Department of Mathematics, UET Lahore}

%
\author{Muhammad Tahir}
\address{Department of Electrical Engineering, UET Lahore}



\begin{abstract}
Combinatorial number system represents a non-negative natural numbers as sum of binomial coefficients. This paper presents an induction proof that there exists unique representation of every non-negative natural number $m$ as sum of $r$ binomial coefficients.
\end{abstract}

\maketitle

\section{Introduction}
Combinadics or combinatorial number system is a mixed radix representation for natural numbers. The place value assigned to a binary digit is a binomial coefficient. By restricting the number of binomial coefficients, say to $r$, any non-negative integer $m$ can be represented uniquely as sum of $r$ binomial coefficients. A different value of $r$ would result in a different combinatorial number system. \\

Combinatorial number system solves the ranking/unranking problem by directly calculating the lexicographic order of a permutation of binary string with $r$ ones.  The bijection between a natural number and its combinatorial representation was first seen by D.H.Lehmer \cite{beckenbach1964}. Enumeration of combinations is a ubiquitous computer science problem with well established algorithms \cite{knuth2005}\cite{butler2011}. Finding combinations and their lexicographic order using combinatorial number system\cite{pascal1887} is also known as combinadics, a term credited to  James McCaffrey\cite{ mccaffrey2004}. \\

This paper presents a simple bijection proof between a number and its combinatorial representation using mathematical induction and the Hockey-Stick identity of the Pascal's triangle. After stating the combinadic theorem and helping lemmas, section-2 proves the existence of combinatorial representation for a non-negative natural number. Section-3 proves the uniqueness of such a representation.

%
%

\begin{theorem}[Combinadics]
$\forall m\geq0, m \in \mathbb{N}_{\geq 0}$, there exist unique $ r, C_r,\cdots,C_i,\cdots,C_1$ such that $C_i \geq 0$, and  $C_j>C_i$  for $ j>i$ and
\begin{eqnarray}
	\label{eqn:CSum}
	m&=& \sum_{i=1}^r \binom{C_i}{i} \\
	&=&\binom{C_r}{r}+\binom{C_{r-1}}{r-1}+\cdots+\binom{C_i}{i}+\cdots+\binom{C_2}{2}+\binom{C_1}{1} \nonumber
\end{eqnarray}

\end{theorem}

We use the following definition for binomial coefficients for two non-negative integers $n,r\in\mathbb{N}_{\geq0}$

\begin{definition}

\begin{eqnarray*}
\binom{n}{r}&=&\frac{n!}{r!(n-r)!}~\text{for}~n\geq r\\
 &=&0~\text{otherwise}
\end{eqnarray*}
\end{definition}

To represent a zero in $r$ binomial coefficient system, equation-\ref{eqn:CSum} reduces to
\begin{equation}
\label{eqn:zero}
0 = \binom{r-1}{r}+\cdots+\binom{i-1}{i}+\cdots+\binom{1}{2}+\binom{0}{1}
\end{equation}

This is the uniqure representation of zero with the given restrictions on $C_i$'s. For subsequent argument we will make use of this zero. For brevity we use the notation $0^r$ to represent a zero consisting of $r$ binomial terms as per equation-\ref{eqn:zero}.

\begin{lemma}
\label{the:hockey}
We prove the following Hockey-Stick Identity for Pascal's triangle 
	\begin{equation}
		\label{eqn:hockey}
		\binom{n+1}{r} = \sum_{i=0}^{r} \binom{n-r+j}{j}	
	\end{equation}

\end{lemma}

\begin{proof}
We prove it using the basic property of the Pascal's triangle. That is every entry is the sum of two entries in the preceding row, one on the top and the other on the top-left of the current entry.
\begin{equation}
\label{eqn:pascal}
 \binom{n+1}{r}=\binom{n}{r}+\binom{n}{r-1}	
 \end{equation}

We get equation-\ref{eqn:hockey} by applying equation-\ref{eqn:pascal} recursively and repeating the process till $r$ is decreased to zero.

\begin{eqnarray*}
\binom{n+1}{r}&=&\binom{n}{r}+\binom{n}{r-1}\\
&=&\binom{n}{r}+\binom{n-1}{r-1}+\binom{n-1}{r-2}\\
&=&\binom{n}{r}+\binom{n-1}{r-1}+\binom{n-2}{r-2}+\binom{n-2}{r-3}\\
&&\quad \vdots \\
&=&\binom{n}{r}+\binom{n-1}{r-1}+\cdots+\binom{n-r+i}{i}+\cdots+\binom{n-r+1}{1}+\binom{n-r}{0}\\
\end{eqnarray*}

\end{proof}

Since $\binom{x}{0}=1$, $\forall x \in \mathbb{N}_{\geq 0}$. The consequence of lemma-\ref{the:hockey} is the next corollary. 
\begin{corollary}
\label{cor:binomSum}
\begin{equation}
\binom{n+r}{r}>\sum_{i=1}^r \binom{n+i-1}{i}
\end{equation}

\end{corollary}

Also for  $n_1,n_2 \in \mathbb{N}_{\geq 0}$ and $n_2>n_1$, lemma-\ref{the:hockey} gives 

\begin{corollary}[]
\label{cor:binomCompare}
\begin{equation}
\binom{n_2}{r}>\binom{n_1}{r}
\end{equation}
\end{corollary}

\section{Existense}

\begin{theorem}
\label{the:existence}
$\forall m\geq0, r\geq1$ and $m,r \in \mathbb{N}_{\geq 0}$,  $\exists C_r,\cdots,C_i,\cdots,C_1$ such that $C_i \geq 0$, and  $C_j>C_i$  for $ j>i$ and
\begin{eqnarray}
	\label{eqn:CSum2}
	m&=& \sum_{i=1}^r \binom{C_i}{i} \\
	&=&\binom{C_r}{r}+\binom{C_{r-1}}{r-1}+\cdots+\binom{C_i}{i}+\cdots+\binom{C_2}{2}+\binom{C_1}{1} \nonumber
\end{eqnarray}

\end{theorem}

\begin{proof}

We prove it using mathematical induction.

\textbf{Base Case:} Take $m=0$ as the base case. 

\begin{equation}
0 = \binom{r-1}{r}+\cdots+\binom{i-1}{i}+\cdots+\binom{1}{2}+\binom{0}{1}
\end{equation}
Both LHS and RHS are zero without violating the conditions on $C_i$.\\

\textbf{Induction Hypothesis:}

Each $m \leq k$ has a representation as sum of $r$ binomial coefficients.

\begin{eqnarray}
	\label{eqn:CSum3}
	k&=& \sum_{i=1}^r \binom{C_i}{i} \\
	&=&\binom{C_r}{r}+\binom{C_{r-1}}{r-1}+\cdots+\binom{C_2}{2}+\binom{C_1}{1} \nonumber
\end{eqnarray}

We need to show that $k+1$ can be written as sum of $r$ binomial coefficients.

\textbf{Induction Step:} Add $1$ to both sides of equation-\ref{eqn:CSum2}.

\begin{eqnarray}
	k+1&=& \sum_{i=1}^r \binom{C_i}{i}+1 \\
	&=&\binom{C_r}{r}+\binom{C_{r-1}}{r-1}+\cdots+\binom{C_2}{2}+\binom{C_1}{1}+1 \nonumber
\end{eqnarray}

Consider the first $j$ $C_i$'s that are consective. That means $C_{l+1}=C_l+1$ for $1\leq l<j$

\begin{eqnarray}
\label{eqn:CSumP1}
k+1&=& \sum_{i=1}^r \binom{C_i}{i}+1\\ \nonumber
&=&\binom{C_r}{r}+\binom{C_{r-1}}{r-1}+\cdots+\binom{C_1+j+\alpha}{j+1}+   \\
&&\left\{ \binom{C_1+j-1}{j}+\cdots+\binom{C_1+1}{2}+\binom{C_1}{1}+1 \right\} \nonumber
\end{eqnarray}

Where $\alpha\geq 1$. The part enclosed in curly brackets in equation-\ref{eqn:CSumP1} can be simplified using the Hockey-Stick Identity of equation-\ref{eqn:hockey2}.

	\begin{equation}
		\label{eqn:hockey2}
		\binom{n+r}{r} = 1+\sum_{i=1}^{r} \binom{n+i-1}{i}	
	\end{equation}

The equation-\ref{eqn:CSumP1} now becomes,
\begin{eqnarray}
\label{eqn:CSumP3}
k+1&=&\binom{C_r}{r}+\binom{C_{r-1}}{r-1}+\cdots+\binom{C_1+j+\alpha}{j+1}+   \binom{C_1+j}{j} 
\end{eqnarray}

The equation-\ref{eqn:CSumP3} has $r-j+1$ sum terms instead of $r$ terms. To complete the number of terms we add $0^{j-1}=\binom{j-2}{j-1}+\binom{j-3}{j-2}+\cdots+\binom{1}{2}+\binom{0}{1}$ to its RHS.

\begin{eqnarray}
\label{eqn:CSumP4}
k+1&=& \sum_{i=j+1}^{r} \binom{C_i}{i}+\binom{C_1+j}{j}+\left\{ \sum_{i=1}^{j-1} \binom{i-1}{i} \right\} \\ \nonumber
&=&\binom{C_r}{r}+\binom{C_{r-1}}{r-1}+\cdots+\binom{C_{j+1}}{j+1}+\binom{C_1+j}{j}+   \\
&&\left\{ \binom{j-2}{j-1}+\binom{j-3}{j-2}+\cdots+\binom{1}{2}+\binom{0}{1} \right\} \nonumber
\end{eqnarray}

In equation-\ref{eqn:CSumP4} the part in curly brackets is zero.

Q.E.D.

\end{proof}

\section{Uniqueness}
\begin{theorem}
\label{the:uniqueness}
$\forall m\geq0, r\geq1$ and $m,r \in \mathbb{N}_{\geq 0}$,  we have unique $C_r,\cdots,C_i,\cdots,C_1$ such that $C_i \geq 0$, and  $C_j>C_i$  for $ j>i$ and
\begin{eqnarray}
	\label{eqn:CSum4}
	m&=& \sum_{i=1}^r \binom{C_i}{i} \\
	&=&\binom{C_r}{r}+\binom{C_{r-1}}{r-1}+\cdots+\binom{C_i}{i}+\cdots+\binom{C_2}{2}+\binom{C_1}{1} \nonumber
\end{eqnarray}
\end{theorem}

\begin{proof}
Let $n\in \mathbb{N}_{\geq 0}$

Let $A$ and $B$ be two distinct combinatorial representations of $m$ in $r$ binomial coefficients.

\[ \sum A=\binom{a_r}{r}+\cdots+\binom{a_2}{1}+\binom{a_1}{1}\]

\[ \sum B=\binom{b_r}{r}+\cdots+\binom{b_2}{1}+\binom{b_1}{1}\]

Consider sets $A'=A-B$ and $B'=B-A$, carrying elements of $A$ and $B$ that are not present in the other set.

Since $A$ and $B$ have equal sums:
\begin{equation}
	\label{eqn:equalSum}
	\sum A' = \sum B'
\end{equation}

Without loss of generality assume that $A'$ is empty.

To have the same sum with non-negative coefficients, $B'$ should also be empty. However because $A\neq B$, it follows that both $A'$ and $B'$ must be non-empty.\\

Suppose $\binom{C_A}{r}$ and $\binom{C_B}{r}$ be the largest coefficients in $A'$ and $B'$ respectively. From corollary-\ref{cor:binomCompare}, for non-negative integers $j,r,k \in \mathbb{N}_{\geq 0}$, $ \binom{j}{r}>\binom{k}{r}$ for $j>k$

Since $A'$ and $B'$ carry no common elements, therefore $C_A \neq C_B$ .

Without loss of generality suppose that $C_A<C_B$.

From corollary-\ref{cor:binomSum} it follows

\[ \sum A' < \binom{C_{A+1}}{r}\]
and so:
\[ \sum A' < \binom{C_{B}}{r}\]

\[ n \leq \sum A' < \binom{C_{A+1}}{r} \leq \binom{C_B}{r} \leq n  \]

But from equation-\ref{eqn:equalSum}:
\[ \sum A' = \sum B' \]

From this contradiction, the case that $A'$ and $B'$ are non-empty is not possible.

Therefore $A' = B' = \emptyset$ and so $A=B$.

Thus the combinatorial representation is unique.
\end{proof}

\bibliographystyle{amsplain}

\begin{thebibliography}{10}


\bibitem{beckenbach1964} Beckenbach, Ed. E. F. (1964), "Applied Combinatorial Mathematics", pp.27−30.

\bibitem{knuth2005} Knuth, D. E. (2005), "Generating All Combinations and Partitions", The Art of Computer Programming, 4, Fascicle 3, Addison-Wesley, pp. 5−6, ISBN 0-201-85394-9. 

\bibitem{butler2011} Butler, Jon T., and Tsutomu Sasao. "Index to constant weight codeword converter." Reconfigurable Computing: Architectures, Tools and Applications. Springer Berlin Heidelberg, 2011. 193-205.

\bibitem{pascal1887} Pascal, Ernesto (1887), "Giornale di Matematiche", 25, pp. 45−49 

\bibitem{mccaffrey2004} McCaffrey, James (2004), "Generating the mth Lexicographical Element of a Mathematical Combination", Microsoft Developer Network
\end{thebibliography}

\end{document}